\documentclass[a4paper,12pt,centertags]{amsart}
\usepackage[utf8]{inputenc}
\usepackage{mathrsfs}
\usepackage{eulervm}

\usepackage[T1]{fontenc}
\usepackage[a4paper,asymmetric]{geometry}
\usepackage{tikz}
\usepackage[colorlinks,pdfdisplaydoctitle]{hyperref}
\newtheorem{theorem}{Theorem}[section]
\newtheorem{alphtheorem}{Theorem}
\newtheorem{lemma}[theorem]{Lemma}
\newtheorem{proposition}[theorem]{Proposition}
\theoremstyle{definition}
\newtheorem{definition}[theorem]{Definition}
\theoremstyle{remark}
\newtheorem{remark}[theorem]{Remark}

\numberwithin{equation}{section}

\newcommand{\e}{\operatorname{e}}  % neper
\newcommand{\N}{\mathbf{N}}        % natural numbers
\newcommand{\R}{\mathbf{R}}        % real numbers
\newcommand{\wV}{\widetilde{V}}
\hypersetup{%
  pdftitle={Smooth solutions for the dyadic model},
  pdfauthor={D. Barbato, F. Morandin, M. Romito},
  pdfcreator={M. Romito}
}
\begin{document}
  \title{Smooth solutions for the dyadic model}
  \author[D. Barbato]{David Barbato}
    \address{Dipartimento di Matematica Pura e Applicata\\ Universit\`a di Padova\\ Via Trieste, 63\\ I-35121 Padova, Italia}
    \email{barbato@math.unipd.it}
  \author[F. Morandin]{Francesco Morandin}
    \address{Dipartimento di Matematica\\ Universit\`a di Parma\\ Viale G. P. Usberti, 53/A\\ I-43124 Parma, Italia}
    \email{francesco.morandin@sns.it}
  \author[M. Romito]{Marco Romito}
    \address{Dipartimento di Matematica\\ Universit\`a di Firenze\\ Viale Morgagni 67/a\\ I-50134 Firenze, Italia}
    \email{romito@math.unifi.it}
    \urladdr{\url{http://www.math.unifi.it/users/romito}}
  \thanks{The third  author gratefully acknowledges the support of the Newton
    Institute for Mathematical Sciences in Cambridge (UK), during the program
    \emph{Stochastic partial differential equations}, where part of this work
    has been done.}
  \date{July 20, 2010}
  \begin{abstract}
    We consider the dyadic model, which is a toy model to test issues
    of well--posedness and blow--up for the Navier--Stokes and Euler equations.
    We prove well--posedness of positive solutions of the viscous problem
    in the relevant scaling range which corresponds to Navier--Stokes.
    Likewise we prove well--posedness for the inviscid problem (in a suitable
    regularity class) when the parameter corresponds to the strongest transport
    effect of the non--linearity.
  \end{abstract}
\maketitle
%%
%%
%%%%%%%%%%%%%%%%%%%%%%
\section{Introduction}

We consider the dyadic model introduced in \cite{FriPav04a,KatPav05} and lately
extensively studied in several variants (viscous~\cite{FriPav04b,Che08,CheFri09},
inviscid~\cite{KisZla05,Wal06,CheFriPav07,BarFlaMor08,BarFlaMor09b} and
stochastically forced~\cite{BarFlaMor09a}).

The dyadic model has been studied as a \emph{toy model} for the Euler and
Navier--Stokes equations as it enjoys the main features of the differential
models, such as energy conservation, while having a much simpler mathematical
structure. Here we focus on regularity and well--posedness for positive
solutions to the viscous~\eqref{e:dyadic_viscous} and to the inviscid
problem~\eqref{e:dyadic_inviscid}.

%%
%%%%%%%%%%%%%%%%%%%%%%%%%%%%%%%%
\subsection{The viscous problem}

Let $\nu>0$, $\beta>0$ and consider
\begin{equation}\label{e:dyadic_viscous}
  \begin{cases}
    \dot X_n
      = -\nu\lambda_n^2 X_n
        + \lambda_{n-1}^\beta X_{n-1}^2
        -\lambda_n^\beta X_n X_{n+1},\\
    X_n(0)=x_n,
  \end{cases}
  \qquad n\geq1\quad t\geq0
\end{equation}
where $\lambda_0=0$, $\lambda_n = \lambda^n$ and $\lambda=2$.
We assume that $x_n\geq0$ and this implies (see~\cite{Che08}) that the solution
remains positive at all times. The parameter $\beta$ measures the relative
strength of the dissipation versus the non--linearity.
The range of values $\beta\in(2,\tfrac52]$ is essentially the one corresponding,
within the simplification of the model, to the three dimensional Navier-Stokes
equations. The range arises from scaling arguments applied to the nonlinear
term, we refer to \cite{CheFri09} for further details.

If $\beta\leq2$ the non linear term is dominated by the dissipative one, in this
case Cheskidov~\cite{Che08} proved existence of regular global solutions using
classical techniques, while if $\beta>3$ the non--linearity is too strong and
all solutions with large enough initial condition develop a blow--up~\cite{Che08}.

The two results above are based on ``energy methods'' and do not cover the
range $\beta\in(2,\tfrac52]$, where it becomes crucial to understand how the
structure of the non--linearity drives the dynamics. The method proposed here
(which is reminiscent of a technique used in the context of fluid mechanics
in \cite{MatSin99}) is based on purely dynamical systems techniques.

In order to prove well--posedness of the viscous problem, we identify
a minimal condition that implies smoothness of solutions
(Proposition~\ref{p:viscous_smooth}). The main idea then is to show the existence
of an invariant region for the vector $(X_n,X_{n+1})$ by a dynamical argument
(Lemma~\ref{l:invariant}) which provides the minimal condition. We are led
to the following result.
\begin{alphtheorem}\label{t:main_viscous}
Let $\beta\in(2,\tfrac52]$, then for every initial condition $(x_n)_{n\geq1}$
such that
\[
  x_n\geq0\qquad\text{for all }n\geq1,
  \qquad\text{and}\qquad
  \sum_{n=1}^\infty x_n^2<\infty,
\]
there exists a unique solution to problem~\eqref{e:dyadic_viscous}, which
is smooth, that is
\[
  \sup_{n\geq1} \bigl(\lambda_n^\gamma X_n(t)\bigr)<\infty
\]
for all $\gamma>0$ and $t>0$.
\end{alphtheorem}
%%
%%%%%%%%%%%%%%%%%%%%%%%%%%%%%%%%%
\subsection{The inviscid problem}

It turns out that the invariant region provided by Lemma~\ref{l:invariant} is
independent of the viscosity. This allows to consider the inviscid problem
\begin{equation}\label{e:dyadic_inviscid}
  \begin{cases}
    \dot X_n
      = \lambda_{n-1}^\beta X_{n-1}^2
        -\lambda_n^\beta X_n X_{n+1},\\
    X_n(0)=x_n,
  \end{cases}
  \qquad n\geq1\quad t\geq0.
\end{equation}
It is known that there are local in time regular solutions (namely, with strong
enough decay in $n$) and that there is a finite time blow--up, that is the
quantity
\[
  \sum_{n=1}^\infty \bigl( \lambda_n^{\frac\beta3} X_n(t) \bigr)^2
    \nearrow\infty
\]
when $t$ approaches a finite time \cite{KatPav05,FriPav04a}. Our result gives
a different picture, as we prove that the dynamics generated by \eqref{e:dyadic_inviscid}
is well--posed in a larger space. The correct interpretation to both results is
that the condition above involving the blowing up quantity does not provide
the natural space for the solutions of the inviscid problem. Indeed, a
$\lambda_n^{-\beta/3}$ decay is borderline for the conservation of energy (which
does not holds rigorously for weaker decay, a proof for $\beta\leq 1$ is given
in \cite{BarFlaMor08}).

To support the physical validity of the solutions we consider, we also prove
that the global solution we have found is the unique vanishing viscosity
limit. The main result for \eqref{e:dyadic_inviscid} is given in full details
as follows.
\begin{alphtheorem}\label{t:main_inviscid}
Let $\beta=\tfrac52$ and let $x=(x_n)_{n\geq1}$ with $x_n\geq0$ for
all $n\geq1$ and
\[
  \sup_{n\geq1} \bigl(\lambda_n^\gamma x_n\bigr) < \infty,
\]
for some $\gamma>\tfrac12$ close enough to $\tfrac12$.
Then there is a global in time solution $X=(X_n)_{n\geq1}$ to
\eqref{e:dyadic_inviscid} with initial condition $x$ such that
\begin{equation}\label{e:inviscid_bound}
  \sup_{t\geq0}\bigl(\sup_{n\geq1} \lambda_n^\gamma X_n(t)\bigr) < \infty,
\end{equation}
which is unique in the class of solutions satisfying the bound
\eqref{e:inviscid_bound} above.

Moreover, $X$ is the unique vanishing viscosity limit. More precisely, if
$X^{[\nu]}$ is the solution to the viscous problem \eqref{e:dyadic_viscous}
with viscosity $\nu$ and with initial condition $x$, then
\[
  X_n^{[\nu]} \longrightarrow X_n,
    \qquad n\geq1,
\]
as $\nu\to0$, uniformly in time on compact sets.
\end{alphtheorem}

The paper is organised as follows. In Section~\ref{s:region} we prove the
fundamental invariant region lemma with a dynamical systems technique.
The well--posedness of the viscous problem is established
in Section~\ref{s:viscous}, while the vanishing viscosity limit and the
inviscid problem are analysed in Section~\ref{s:inviscid}.
%%
%%
%%
%%%%%%%%%%%%%%%%%%%%%%%%%%%%%%%%%%%%%%%%%%%%%%%%%%%%
\section{The invariant region lemma}\label{s:region}

In this section we prove the key result of the paper. Let $(X_n)_{n\geq1}$
be a solution to problem~\eqref{e:dyadic_viscous} on a time interval $[0,T]$.
In view of Proposition~\ref{p:viscous_smooth} below, it is natural to apply the
following \emph{change of variables}
\[
  Y_n = \lambda_n^{\beta - 2 + \epsilon} X_n,
\]
where $\epsilon>0$ will be chosen suitably in the proof of the lemma below.
A straightforward computation shows that $(Y_n)_{n\geq1}$ solves
\begin{equation}\label{e:Yeq}
  \begin{cases}
    \dot Y_n
      = -\nu\lambda_n^2 Y_n
        + \lambda_{n-1}^{2-\epsilon} \lambda^{\beta-2+\epsilon} Y_{n-1}^2
        -\lambda_n^{2-\epsilon} \lambda^{2-\beta-\epsilon} Y_n Y_{n+1},\\
    Y_n(0)=y_n,
  \end{cases}
\end{equation}
for $n\geq1$ and $t\in[0,T]$, where clearly $y_n = \lambda_n^{\beta-2 + \epsilon} X_n(0)$
for all $n\geq1$.

For technical reasons we consider a finite dimensional (truncated) version
for the equations for $Y$. For every $N\geq1$ let $(Y_n^{(N)})_{1\leq n\leq N}$
be the solution to
\begin{equation}\label{e:YNeq}
  \begin{cases}
    \dot Y_n^{(N)}
      = -\nu\lambda_n^2 Y_n^{(N)}
        + \lambda_{n-1}^{2-\epsilon} \lambda^{\beta-2+\epsilon} \bigl(Y_{n-1}^{(N)}\bigr)^2
        - \lambda_n^{2-\epsilon} \lambda^{2-\beta-\epsilon} Y_n^{(N)} Y_{n+1}^{(N)},\\
%    \dot Y_N^{(N)}
%      = -\nu\lambda_N^2 Y_N^{(N)}
%        + \lambda_{N-1}^{2-\epsilon} \lambda^{\beta-2+\epsilon} \bigl(Y_{N-1}^{(N)}\bigr)^2
%        - \lambda_N^{2-\epsilon} \lambda^{2-\beta-\epsilon} \bigl(Y_N^{(N)}\bigr)^2,\\
    Y_n(0)=y_n,
  \end{cases}
\end{equation}
for $n=1,\dots,N$, where for the sake of simplicity we have set $Y_0^{(N)}=0$
and $Y_{N+1}^{(N)} = Y_N^{(N)}$, so to avoid writing the border equations
in a different form.
Let us now introduce the region $A$ of $\R^2$ that will be
invariant for the vectors $(Y_n^{(N)},Y_{n+1}^{(N)})$,
\[
  A := \{ (x,y)\in\R^2 : 0\leq x\leq 1, h(x)<y<g(x) \},
\]
where the functions $h$ and $g$ that provide the lower and upper bound of $A$
are defined as
\[
  g(x)
    = \min\{ mx + \theta, 1 \},
  \qquad\qquad
  h(x) =
    \begin{cases}
      0                                                   &\quad x\leq\delta,\\
      c\bigl(\frac{x-\delta}{1-\delta}\bigr)^{\lambda^2}  &\quad x>\delta.
    \end{cases}
\]
\begin{figure}[ht]
  \begin{tikzpicture}[x=40mm,y=40mm,domain=0.1:1,smooth]
    % area condizione iniziale
    \fill[color=lightgray] (0,0) -- (0.1,0) -- (0.1,0.1) -- (0,0.1);
    % assi
    \draw[->,thin] (0,0) -- (1.5,0) node [right] {\scriptsize $Y_n$};
    \draw[->,thin] (0,0) -- (0,1.1) node [above] {\scriptsize $Y_{n+1}$};
    % A
    \draw[thick] (0,0) -- (0, 0.6) node[left] {\scriptsize $\theta$} --%
       (0.53,1) -- (1,1) -- (1,0.66) node[right] {\scriptsize $c$};
%    \draw[thick] plot (\x,{(\x-0.1)*(\x-0.1)*(\x-0.1)*(\x-0.1)});
    \draw[thick] plot (\x,{.82*(\x-0.1)*(\x-0.1)});
    \draw[thick] (0.1,0) node[below] {\scriptsize $\delta$} -- (0,0);
    % label e linee tratteggiate
    \draw[densely dotted] (0,1) node[left] {\scriptsize $1$} -- (0.53,1)%
       (1,0) node[below] {\scriptsize $1$} -- (1,0.66);
    \draw (0.5,0.5) node {\large $A$};
    % normali
    \draw[->] (0,0.3)     -- (0.1,0.3) node[right] {\Tiny $n_1$};
    \draw[->] (0.26,0.8)  -- (0.33,0.72) node[right] {\Tiny $n_2$};
    \draw[->] (0.77,1)    -- (0.77,0.9) node[left] {\Tiny $n_3$};
    \draw[->] (1,0.84)    -- (0.9,0.84) node[below, text centered] {\Tiny $n_4$};
    \draw[->] (.67,0.27)  -- (0.6,0.33) node[left] {\Tiny $n_5$};
    \draw[->] (0.05,0)    -- (0.05,0.1) node[right] {\Tiny $n_6$};
  \end{tikzpicture}
  \caption{The invariant region}
  \label{f:invariant}
\end{figure}
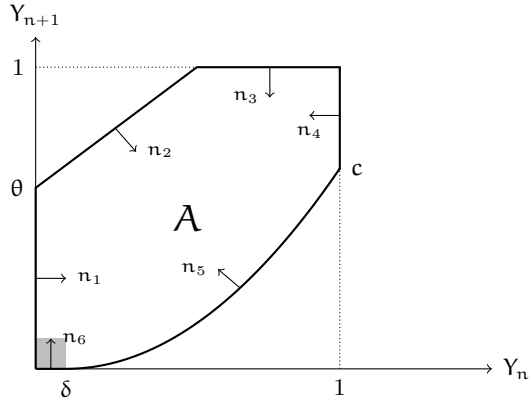
\begin{lemma}\label{l:invariant}
There exist $\delta\in(0,1)$, $c\in(0,1)$, $\theta\in(0,1)$, $m>0$ and $\epsilon>0$
such that for every $\beta\in(2,\tfrac52]$ and $\nu\geq0$ the following statement
holds true: if $N\geq1$ and if $(y_n,y_{n+1})\in A$ for all $n\leq N$, then
$(Y_n^{(N)}(t),Y_{n+1}^{(N)}(t))\in A$ for all $n=1,\dots,N$ and $t\geq0$,
where $(Y_n^{(N)})_{1\leq n\leq N}$ is the solution to \eqref{e:YNeq} with
initial condition $(y_n)_{1\leq n\leq N}$.
\end{lemma}
\begin{proof}
For simplicity we drop the superscript ${}^{(N)}$ along this proof. Since the
pairs $(Y_n,Y_{n+1})_{1\leq n\leq N}$ satisfy a finite dimensional system of
differential equation, it is sufficient to show that the derivative in time
of $(Y_n,Y_{n+1})$ points inward on the border of $A$ when $(Y_n,Y_{n+1})\in A$
for each $n=1,\dots,N$ or, equivalently, that the scalar product with the
inward normal of the border of $A$ with the vector field
\[
  \mathfrak{B}
    = \begin{pmatrix}
        \dot Y_n\\
        \dot Y_{n+1}  
      \end{pmatrix}
    = \nu\lambda_n^2
         \begin{pmatrix}
           -Y_n\\
           -\lambda^2 Y_{n+1}
         \end{pmatrix}
      + \lambda^{\beta-4+2\epsilon} \lambda_n^{2-\epsilon}
         \begin{pmatrix}
           Y_{n-1}^2 - \lambda^{6-2\beta-3\epsilon} Y_n Y_{n+1}\\
           \lambda^{2-\epsilon}(Y_n^2 - \lambda^{6-2\beta-3\epsilon} Y_{n+1} Y_{n+2})
         \end{pmatrix}.
\]
is positive when $(Y_n,Y_{n+1})\in A$ for all $n=1,\dots,N$.
The set $A$ is convex, hence we can consider separately the viscous and the
inviscid contribution to $\mathfrak{B}$.

We start with the viscous part, which we denote by $\mathfrak{B}_v$ (we neglect
the multiplicative constant $\nu\lambda_n^2$) and we denote the inward normals
as in Figure~\ref{f:invariant}. The scalar product of $\mathfrak{B}_v$ with
each $\vec{n}_1$, $\vec{n}_3$, $\vec{n}_4$, $\vec{n}_6$ on the respective
pieces of the border of $A$ is clearly positive, as
\[
\begin{gathered}
  \mathfrak{B}_v\cdot\vec{n}_1 = -Y_n = 0,
    \qquad
  \mathfrak{B}_v\cdot\vec{n}_3 = \lambda^2 Y_{n+1} = \lambda^2,\\
  \mathfrak{B}_v\cdot\vec{n}_4 = Y_n = 1,
    \qquad
  \mathfrak{B}_v\cdot\vec{n}_6 = -\lambda^2 Y_{n+1} = 0,
\end{gathered}
\]
so we are left with the last two cases, in which, for simplicity, we set
$x = Y_n$. First,
\[
  \mathfrak{B}_v\cdot\vec{n}_2
    = -x g'(x) + \lambda^2 Y_{n+1}
    = -x g'(x) + \lambda^2 g(x)
    = m (\lambda^2 - 1)x + \theta\lambda^2
    >0,
\]
then
\[
  \mathfrak{B}_v\cdot\vec{n}_5
    = x h'(x) - \lambda^2 Y_{n+1}
    = x h'(x) - \lambda^2 h(x)
    = \frac{c\delta\lambda^2}{1-\delta}\Bigl(\frac{x-\delta}{1-\delta}\Bigr)^{\lambda^2}
    \geq0.
\]

We consider now the inviscid term, that we denote by $\mathfrak{B}_i$ (and
again we neglect the irrelevant multiplicative factor). Again we set $x = Y_n$
and, for simplicity, $\gamma = 6 - 2\beta - 3\epsilon$. We consider first
the easy terms,
\[
\begin{gathered}
  \mathfrak{B}_i\cdot\vec{n}_1
    = Y_{n-1}^2 - \lambda^\gamma x Y_{n+1}
    = Y_{n-1}^2
    \geq0,\\
  \mathfrak{B}_i\cdot\vec{n}_6
    = \lambda^{2-\epsilon}(x^2 - \lambda^\gamma Y_{n+1} Y_{n+2})
    = \lambda^{2-\epsilon} x^2
    \geq0.
\end{gathered}  
\]
Next, we consider the piece of the border of $A$ corresponding to $\vec{n}_3$.
Here $Y_{n+1}=1$ and $x\leq1$, moreover since $(Y_{n+1},Y_{n+2})\in A$, it
follows that $Y_{n+2}\geq c$, hence
\[
  \mathfrak{B}_i\cdot\vec{n}_3
    = \lambda^{2-\epsilon}(\lambda^\gamma Y_{n+1} Y_{n+2} - x^2)
    \geq \lambda^{2-\epsilon}(\lambda^\gamma c - 1).
\]
The term on the right hand side in the formula above is positive if we
choose $\lambda^\gamma c = 1$. Likewise on the piece
corresponding to $\vec{n}_4$ we have $x = Y_n = 1$, $Y_{n+1}\geq c$ and
$Y_{n-1}\leq1$, hence
\[
  \mathfrak{B}_i\cdot\vec{n}_4
    = \lambda^\gamma x Y_{n+1} - Y_{n-1}^2
    \geq \lambda^\gamma c - 1
    \geq 0.
\]

We are left with the two challenging inequalities, that we are going to analyse.
The first is on the piece of boundary corresponding to $\vec{n_2}$, where
 we have $Y_{n+1} = g(x)$, and, since $(Y_{n+1},Y_{n+2})\in A$,
$Y_{n+2}\geq h(Y_{n+1}) = h(g(x)) = c\bigl(\frac{mx+\theta-\delta}{1-\delta}\bigr)^{\lambda^2}$,
if we choose $\theta\geq\delta$. Hence, using the fact that $\lambda^\gamma c=1$
and that $\gamma\leq 2-3\epsilon$,
\begin{equation}\label{e:trouble1}
\begin{aligned}
  \mathfrak{B}_i\cdot\vec{n}_2
    & =  g'(x)(Y_{n-1}^2 - \lambda^\gamma x Y_{n+1})
       - \lambda^{2-\epsilon}(x^2 - \lambda^\gamma Y_{n+1} Y_{n+2})\\
    &\geq -\lambda^\gamma x g'(x) g(x)
      - \lambda^{2-\epsilon}\bigl(x^2 - \lambda^\gamma g(x)h(g(x))\bigr)\\
    & = \lambda^{2-\epsilon}(m x + \theta)\Bigl(\frac{m x + \theta - \delta}{1 - \delta}\Bigr)^{\lambda^2}
          - \lambda^{2-\epsilon} x^2 - \lambda^\gamma m x (m x + \theta)\\
    &\geq \lambda^{2-\epsilon}(m x + \theta)\Bigl(\frac{m x + \theta - \delta}{1 - \delta}\Bigr)^{\lambda^2}
          - \lambda^{2-\epsilon} x^2 - \lambda^{2-3\epsilon} m x (m x + \theta).
\end{aligned}
\end{equation}
This last expression depends on $x$ but not on $\beta$ and it is sufficient to
show that it is non--negative for $x\in[0,\tfrac{1-\theta}{m}]$. This will be
done later by a suitable choice of the parameters.

Prior to this, we consider the second inequality, on the piece corresponding to
$\vec{n}_5$. Here  we have that $Y_{n+1}=h(x)$ and $Y_{n-1}\leq 1$, and, since
$(Y_{n+1},Y_{n+2})\in A$, $Y_{n+2}\leq g(Y_{n+1}) = g(h(x))\leq m h(x) + \theta$.
Therefore, since $x\bigl(\tfrac{x-\delta}{1-\delta}\bigr)^{\lambda^2}\leq 1$
and $\gamma\geq1-3\epsilon$, hence $\lambda^{-\gamma}\leq\lambda^{3\epsilon-1}$,
\begin{equation}\label{e:trouble2}
\begin{aligned}
  \mathfrak{B}_i\cdot\vec{n}_5
    & = \lambda^{2-\epsilon}(x^2 - \lambda^\gamma Y_{n+1} Y_{n+2})
        - h'(x)(Y_{n-1}^2 - \lambda^\gamma x Y_{n+1})\\
    &\geq \lambda^{2-\epsilon} \bigl(x^2 - \lambda^\gamma h(x) g(h(x))\bigr)
          - h'(x)(1 - \lambda^\gamma x h(x))\\
    & =   \lambda^{2-\epsilon}\Bigl[x^2 - \theta\Bigl(\frac{x-\delta}{1-\delta}\Bigr)^{\lambda^2}\Bigr]
        - m\lambda^{2-\gamma-\epsilon} \Bigl(\frac{x-\delta}{1-\delta}\Bigr)^{2\lambda^2} + {}\\
    &\quad - \frac{\lambda^{2-\gamma}}{1-\delta} \Bigl(\frac{x-\delta}{1-\delta}\Bigr)^{\lambda^2-1}
              \Bigl[1 - x\Bigl(\frac{x-\delta}{1-\delta}\Bigr)^{\lambda^2}\Bigr]\\
    &\geq \lambda^{2-\epsilon}\Bigl[x^2 - \theta\Bigl(\frac{x-\delta}{1-\delta}\Bigr)^{\lambda^2}\Bigr]
        - m\lambda^{1+2\epsilon} \Bigl(\frac{x-\delta}{1-\delta}\Bigr)^{2\lambda^2} + {}\\
    &\quad - \frac{\lambda^{1+3\epsilon}}{1-\delta} \Bigl(\frac{x-\delta}{1-\delta}\Bigr)^{\lambda^2-1}
              \Bigl[1 - x\Bigl(\frac{x-\delta}{1-\delta}\Bigr)^{\lambda^2}\Bigr]\\
\end{aligned}
\end{equation}
for $x\in[\delta,1]$. Also this lower bound does not depend on $\beta$.

Let $\psi_1$ and $\psi_2$ be the right--hand sides of \eqref{e:trouble1}
and \eqref{e:trouble2}, respectively, when $\epsilon=0$, namely,
\[
\begin{gathered}
  \psi_1(x)
    = (m x + \theta)\Bigl(\frac{m x + \theta - \delta}{1 - \delta}\Bigr)^{\lambda^2}
          - x^2 - m x (m x + \theta),\\
  \begin{aligned}
  \psi_2(x)
    & = \lambda\Bigl[x^2 - \theta\Bigl(\frac{x-\delta}{1-\delta}\Bigr)^{\lambda^2}\Bigr]
        - m \Bigl(\frac{x-\delta}{1-\delta}\Bigr)^{2\lambda^2} + {}\\
    &\quad - \frac1{1-\delta} \Bigl(\frac{x-\delta}{1-\delta}\Bigr)^{\lambda^2-1}
             \Bigl[1 - x\Bigl(\frac{x-\delta}{1-\delta}\Bigr)^{\lambda^2}\Bigr].
  \end{aligned}
\end{gathered}
\]
It is sufficient to show that both function have positive minimal values.
Continuity then ensures that the same is true for small $\epsilon$. 
A direct computation shows that both $\psi_1$ and $\psi_2$ are positive with
the choice $\delta=\tfrac1{10}$, $\theta=\tfrac35$, $m=\tfrac{3}{4}$. Figure
~\ref{f:amano}
shows a plot of the two functions.
\end{proof}
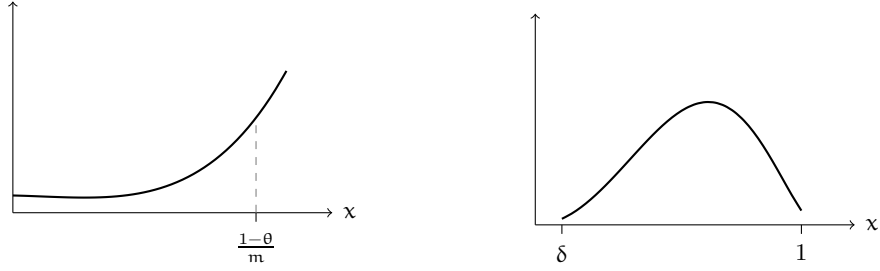
\begin{figure}[ht]
  % prima disuguaglianza
  \begin{tikzpicture}[x=60mm,y=40mm,domain=0:.6,smooth]
     \draw[->,thin] (0,0) -- (.7,0) node [right] {\scriptsize $x$};
     \draw[->,thin] (0,0) -- (0,.7);
     \draw (.533,0) -- (.533,-0.03) node [below] {\scriptsize $\tfrac{1-\theta}{m}$};
     \draw[dashed,gray] (.533,0) -- (.533, .315);
     \draw[thick] plot (\x, %
       {(.75*\x+.6)*(.75*\x+.5)*(.75*\x+.5)*(.75*\x+.5)*(.75*\x+.5)/.6561
       - \x*\x %
       - .75*\x*(.75*\x+.6)});
  \end{tikzpicture}
\hspace{20mm}
  % seconda disuguaglianza
  \begin{tikzpicture}[x=35mm,y=40mm,domain=0.1:1,smooth]
     \draw[->,thin] (0,0) -- (1.2,0) node [right] {\scriptsize $x$};
     \draw[->,thin] (0,0) -- (0,.7);
     \draw (0.1,0) -- (0.1,-0.03) node [below] {\scriptsize $\delta$};
     \draw (1,0) -- (1,-0.03) node [below] {\scriptsize $1$};
     \draw[thick] plot (\x, %
       { 2*\x*\x %
       - 1.83*(\x-.1)*(\x-.1)*(\x-.1)*(\x-.1) %
       - 1.74*(\x-.1)*(\x-.1)*(\x-.1)*(\x-.1)*(\x-.1)*(\x-.1)*(\x-.1)*(\x-.1) %
       - 1.69*(\x-.1)*(\x-.1)*(\x-.1)* %
           (1 - 1.52*\x*(\x-.1)*(\x-.1)*(\x-.1)*(\x-.1)) });
  \end{tikzpicture}
  \caption{The functions $\psi_1$, on the left, and $\psi_2$, on the right.}
  \label{f:amano}
\end{figure}
\begin{remark}
A cleverer choice of the parameters $\delta$, $\theta$, and $m$ might
allow to extend the above result, and in turn the main results of the paper,
to larger values of $\beta$ (although smaller than $3$, due to the blow--up
results in \cite{Che08} and \cite{FriPav04a}).
\end{remark}
%%
%%
%%
%%%%%%%%%%%%%%%%%%%%%%%%%%%%%%%%%%%%%%%%%%%%%%%%%%%%%%%%%%%%%%%%%%%%%%%%
\section{Uniqueness and regularity in the viscous case}\label{s:viscous}

Define
\begin{equation}\label{e:Hspace}
  H
    = \{x=(x_n)_{n\geq1}\subset\R: \|x\|_H^2:=\sum_{n=1}^\infty x_n^2 <\infty\}.
\end{equation}
Following Cheskidov~\cite{Che08} we introduce weak and Leray--Hopf solutions
for~\eqref{e:dyadic_viscous}.
\begin{definition}
A \emph{weak solution} to~\eqref{e:dyadic_viscous} on $[0,T]$ is a sequence of
functions $X=(X_n)_{n\geq1}$ such that $X_n\in C^1([0,T];\R)$ for every $n\geq1$
and~\eqref{e:dyadic_viscous} is satisfied.

A \emph{Leray-Hopf solution} is a weak solution $X$ with values in $H$ and
such that the \emph{energy inequality}
\[
  \|X(t)\|_H^2
  + 2\nu\int_s^t \sum_{n=1}^\infty (\lambda_n X_n(r))^2\,dr
    \leq \|X(s)\|_H^2,
\]
holds for a.~e.~$s$ and all $t>s$.
\end{definition}
The following facts are proved in \cite{Che08},
\begin{itemize}
  \item existence of global in time Leray--Hopf solutions for all initial conditions in $H$,
  \item if the initial condition $(x_n)_{n\geq1}$ is \emph{positive}, namely $x_n\geq0$
    for all $n\geq1$, then every weak solution is a Leray--Hopf solution, stays
    positive for all times and the energy inequality holds for all times,
  \item if $\beta\leq2$, there is a unique Leray--Hopf solution which is smooth,
    for every initial condition in $H$,
  \item if $\beta>3$, then every positive solution (starting from a large
    enough initial condition) cannot be smooth for all times.
\end{itemize}
Our first result is a criterion for uniqueness of positive solutions.
\begin{proposition}[Uniqueness]\label{p:viscous_uniq}
Let $X = (X_n)_{n\geq1}$ be a \emph{positive} solution to \eqref{e:dyadic_viscous}
on $[0,T]$ such that the quantity
\begin{equation}\label{e:quniq}
  \sup_{t\in [0,T],n\geq1}\bigl(\lambda_n^{\beta-3}X_n(t)\bigr)
\end{equation}
is finite. Then $X$ is the unique weak solution with initial condition
$(X_n(0))_{n\geq1}$.

In particular, if $\beta\leq3$, there is a unique weak solution for
any positive initial condition in $H$.
\end{proposition}
\begin{proof}
The proof is a minor variation of the idea in~\cite{BarFlaMor09b}. Denote by
$c_0$ the quantity~\eqref{e:quniq}. Let $Y=(Y_n)_{n\geq1}$ be another
solution with the same initial condition of $X$ and set
$Z_n = Y_n - X_n$, $W_n = X_n + Y_n$, then
\[
  \dot Z_n
    = - \nu\lambda_n^2 Z_n
      + \lambda_{n-1}^\beta Z_{n-1} W_{n-1}
      - \frac12 \lambda_n^\beta (Z_n W_{n+1} + Z_{n+1} W_n).
\]
Fix $N\geq1$ and set $\psi_N(t) = \sum_{n=1}^N \tfrac{1}{2^n} Z_n^2$, then
$\psi_N(0) = 0$ and it is elementary to verify that
\[
  \frac{d}{dt}\psi_N(t) + 2\nu\sum_{n=1}^N\frac{\lambda_n^2}{2^n} Z_n^2
    = - \frac12\sum_{n=1}^N \frac{\lambda_n^\beta}{2^n} Z_n^2 W_{n+1}
      - \frac{\lambda_N^\beta}{2^{N+1}} Z_N Z_{N+1} W_N.
\]
In particular (we recall that $\lambda = 2$ and $\lambda_n = \lambda^n$),
\[
  \begin{aligned}
    \frac{d}{dt}\psi_N(t)
      &\leq -\frac12 \lambda_N^{\beta-1} Z_N Z_{N+1} W_N\\
      &= -\frac12 \lambda_N^{\beta-1} (Y_N^2Y_{N+1} + X_N^2 X_{N+1} - X_{N+1}Y_N^2 - X_N^2 Y_{N+1})\\
      &\leq \frac12 \lambda_N^{\beta-1} (X_{N+1}Y_N^2 + X_N^2 Y_{N+1})\\
      &\leq c_0\lambda_N^2(X_N^2 + Y_N^2 + Y_{N+1}^2),
  \end{aligned}
\]
and so by integrating in time,
\[
  \psi_N(t)
    \leq c_0\int_0^t \lambda_N^2(X_N^2 + Y_N^2 + Y_{N+1}^2)\,ds.
\]
Since $X$ and $Y$ are both Leray--Hopf solutions, the right hand side in the
above inequality converges to $0$ as $N\to\infty$ and in conclusion
$\psi_n(t)=0$ for all $t\geq0$ and all $n\geq1$.
\end{proof}
%%
%%%%%%%%%%%%%%%%%%%%%%%
\subsection{Regularity}

Having the key Lemma~\ref{l:invariant} in hand, the missing step for the
proof of Theorem~\ref{t:main_viscous} is a regularity criterion. The next result
gives a minimal condition of smoothness which is in a way essentially optimal,
as shown in Section~\ref{sss:stationary} below, and which holds for general
(positive and non--positive) initial conditions. Set
\[
  \mathcal{D}^\infty
    = \{ (x_n)_{n\geq1} : \sup_{n\geq1}\bigl(\lambda_n^\gamma|x_n|\bigr)<\infty
         \text{ for all }\gamma>0 \}.
\]
\begin{proposition}\label{p:viscous_smooth}
Let $T>0$ and let $X$ be a solution to~\eqref{e:dyadic_viscous} on $[0,T]$ such that
$X(0)\in\mathcal{D}^\infty$ and
\[
  \lim_{n\to\infty}\Bigl(\sup_{t\in[0,T]}\bigl(\lambda_n^{\beta-2}|X_n(t)|\bigr)\Bigr) = 0.
\]
Then $X(t)\in\mathcal{D}^\infty$ for all $t\in[0,T]$. In particular, the above
condition is verified if there is $\epsilon>0$ such that
\[
  \sup_{n\geq1}\Bigl(\sup_{t\in[0,T]}\bigl(\lambda_n^{\beta-2+\epsilon}|X_n(t)|\bigr)\Bigr)<\infty.
\]
\end{proposition}
\begin{proof}
We can assume without loss of generality that $\lambda_n^{\beta-2}|X_n(t)|\leq c_n$
for all $n\geq1$ and $t\in[0,T]$, with $c_n\downarrow0$. Since
\[
  X_n(t)
    = \e^{-\nu\lambda_n^2t} X_n(0)
      + \int_0^t \e^{-\nu\lambda_n^2(t-s)}\bigl(\lambda_{n-1}^\beta X_{n-1}^2 - \lambda_n^\beta X_n X_{n+1}\bigr)\,ds,
\]
we have that
\[
  |X_n(t)|
    \leq |X_n(0)|
      + \lambda^2 c_{n-1}\int_0^t \lambda_n^2\e^{-\nu\lambda_n^2(t-s)} (|X_{n-1}| + |X_n|)\,ds,
\]
and so for every $\gamma>0$,
\[
  G_n
    \leq \lambda_n^\gamma|X_n(0)|
      + \frac{\lambda^{2 + \gamma}}{\nu} c_{n-1} (G_{n-1} + G_n),
\]
where we have set $G_n = \sup_{t\in[0,T]}(\lambda_n^\gamma|X_n(t)|)$. Hence
there is $n_0$ such that for $n\geq n_0$ we have
$\lambda^{2+\gamma}\nu^{-1} c_{n-1}\leq\tfrac13$ and so
$\sup_{n\geq n_0} G_n<\infty$. The terms $G_n$ for $n\leq n_0$ are bounded
due to the assumption.
\end{proof}
\begin{remark}[Local smooth solutions]\label{r:localsmooth}
The $\lambda_n^{\beta-2}$ decay can be interpreted in terms of local existence
and uniqueness of smooth solutions. Indeed, this decay is critical, in the sense
that only exponents larger or equal than $\beta-2$ allow for local smooth
solutions (for any general quadratic finite--range interaction non--linearity,
without taking the geometry into account). This can be seen in the following
way. Set for $\epsilon>0$
\begin{equation}\label{e:Wspace}
  \mathcal{W}_\epsilon
    = \{ x=(x_n)_{n\geq1} : \|x\|_{\mathcal{W}_\epsilon}:=\sup_{n\geq1}\bigl(\lambda_n^{\beta-2+\epsilon}|x_n|\bigr)<\infty \},
\end{equation}
the result is a standard application of Banach's fixed point theorem to the
map
\[
  \mathcal{F}_n(X)(t)
    = \e^{-\nu\lambda_n^2 t}X_n(0)
      + \int_0^t \e^{-\nu\lambda_n^2(t-s)}
          \bigl(\lambda_{n-1}^\beta X_{n-1}^2 - \lambda_n^\beta X_n X_{n+1}\bigr)\,ds
\]
and the relevant estimate to prove that $\mathcal{F}$ maps a small ball into
itself and is a contraction (for a small enough time interval) is
\begin{multline*}
  \int_0^t \e^{-\nu\lambda_n^2(t-s)}
      \Bigl(\lambda_{n-1}^\beta X_{n-1}Y_{n-1} - \frac12\lambda_n^\beta\bigl(X_n Y_{n+1}+X_{n+1} Y_n\bigr)\Bigr)\,ds\leq\\
    \leq \frac{c_\lambda}{\nu^{1-\frac\epsilon2}} \lambda_n^{2-\beta-\epsilon} T^{\frac\epsilon2}
      \bigl(\sup_{t\leq T}\|X\|_{\mathcal{W}_\epsilon}\bigr)\bigl(\sup_{t\leq T}\|Y\|_{\mathcal{W}_\epsilon}\bigr).
\end{multline*}
The case $\epsilon=0$ (the critical case!) does not allow for small constants
and can be worked out as in~\cite{FujKat64}.
\end{remark}
%%
%%%%%%%%%%%%%%%%%%%%%%%%%%%%%%%%%%%%%%%%%%%%%%%%%%%%%%%%%%%%%%%%%%%%%%%%%%%%%%%%%%
\subsubsection{Stationary solutions and critical regularity}\label{sss:stationary}

In this section we show that the condition given in Proposition~\ref{p:viscous_smooth}
is optimal, by showing that there is a solution to~\eqref{e:dyadic_viscous}
such that the quantity $\sup_{t,n}\bigl(\lambda_n^{\beta-2}|X_n(t)|\bigr)$ is
bounded but the solution is not smooth. The example is provided by a time--stationary
solution. In order to do this in this section (and only in this section) we
shall consider solutions to~\eqref{e:dyadic_viscous} which may have also
non--positive components.

We shall call \emph{stationary solution} any sequence $\gamma=(\gamma_n)_{n\geq1}$
such that
\begin{equation}\label{e:stat}
  \nu\lambda_n^2\gamma_n + \lambda_n^\beta\gamma_n\gamma_{n+1} - \lambda_{n-1}^\beta\gamma_{n-1}^2 = 0,
    \qquad n\geq1,
\end{equation}
\begin{proposition}
  Let $\gamma=(\gamma_n)_{n\geq1}$ be a non-zero stationary solution.
  \begin{itemize}
    \item If there is $n_0\geq1$ such that $\gamma_{n_0}=0$, then
      $\gamma_n=0$ for all $n\leq n_0$.
    \item Let $n_0$ be the first index such that $\gamma_{n_0}\neq0$.
      Then $\gamma_n<0$ for all $n>n_0$.
    \item Let $n_0$ be the first index such that $\gamma_{n_0}\neq0$.
      Then there is $c>0$ such that $\lambda_n^{\beta-2}|\gamma_n|\geq c$,
      for all $n\geq n_0$.
  \end{itemize}
\end{proposition}
\begin{proof}
Multiply~\eqref{e:stat} by $\gamma_n$ and sum up to $N$ to obtain
\[
  \nu\sum_{n=1}^N \lambda_n^2\gamma_n^2 + \lambda_N^\beta\gamma_N^2\gamma_{N+1} = 0,
    \qquad N\geq1.
\]
The first two properties follow from this equality. For the third property,
\eqref{e:stat} implies that
\[
  \gamma_{n+1}
    = \frac{\gamma_{n-1}^2}{\lambda^\beta\gamma_n} - \nu\lambda_n^{2-\beta}
    \leq  - \nu\lambda_n^{2-\beta},
\]
since all $\gamma_n$ are negative.
\end{proof}
Hence a stationary solutions can decay at most as the critical profile which
is borderline in Proposition~\ref{p:viscous_smooth}. So the existence of
a stationary solutions shows that the condition of Proposition~\ref{p:viscous_smooth}
is optimal. Moreover, if the stationary solution is in $H$, this provide an
example of two weak solutions with the same initial condition (the stationary
solution and the Leray--Hopf solution).

We look now for a stationary solution $(\gamma_n)_{n\geq1}$. Set
$u = \lambda^{2\beta-6}$ (notice that $u<1$ for $\beta<3$) and
$\gamma_n = -\nu\lambda_{n-1}^{2-\beta}a_n$. Then
\[
\begin{cases}
  a_1 (a_2 - 1) = 0,\\
  a_n a_{n+1} = a_n + u a_{n-1}^2,
    &\quad n\geq2.
\end{cases}
\]
One can show that if $u<\tfrac13$, then there are infinitely many stationary
solutions such that $0<c_1\leq \lambda_n^{\beta-2}|\gamma_n|\leq c_2$. Indeed,
consider $a_1\geq0$ and $a_2=1$ and set
\[
  A
    = \frac{1}{2u}\Bigl(1-\sqrt{\frac{1-3u}{1+u}}\Bigr),
      \qquad
  B 
    = \frac{1}{2u}\Bigl(1+\sqrt{\frac{1-3u}{1+u}}\Bigr).
\]
It is easy to verify that if $a_{n-1}, a_n\in [A,B]$, then $a_{n+1}\in[A,B]$,
and so one needs only to find values of $a_1$ such that the sequence
$(a_n)_{n\geq1}$ ends up in $[A,B]$. This requires a few computations which
are not relevant for the paper and are omitted.
%On the other hand if $u\geq\tfrac13$,
%one should find a stationary solution using computations similar to those
%developed in \cite{BarFlaMor08} for self--similar solutions.
%%
%%%%%%%%%%%%%%%%%%%%%%%%%%%%%%%%%%%%%%%%%%%%%%%%%%
\subsection{Proof of Theorem~\ref{t:main_viscous}}

We have now all ingredients for the proof of the main theorem concerning
the viscous case.
\begin{proof}[Proof of Theorem~\ref{t:main_viscous}]
Let $x\in H$ be positive and let $X=(X_n)_{n\geq1}$ be the unique weak solution
starting at $x$. To prove the theorem, it is sufficient to show the following
two claims,
\begin{enumerate}
  \item for some $\epsilon>0$ and for every $t_0>0$ the quantity
    $\sup_{n\geq1} (\lambda^{\beta-2+\epsilon}X_n(t_0))$ is finite and
    \[
      \sup_{t\geq t_0}\Bigl(\sup_{n\geq1} \bigl(\lambda^{\beta-2+\epsilon}X_n(t)\bigr)\Bigr)
        \leq\frac1\delta \sup_{n\geq1} \bigl(\lambda^{\beta-2+\epsilon}X_n(t_0)\bigr)
    \]
    for every $n\geq1$ and $t\geq t_0$, where $\delta$ is the constant in
    Lemma~\ref{l:invariant}.
  \item if $\sup_{n\geq1} (\lambda^{\beta-2+\epsilon}X_n(t_0))$ is finite,
    then there exists $t_0'>t_0$ such that $X_n$ is smooth in $(t_0,t_0']$.
\end{enumerate}
Indeed, if for $t_0>0$ the first claim holds true, then the second claim applies
and the solution satisfies the assumptions of Proposition~\ref{p:viscous_smooth}
for any initial time $t>t_0$ sufficiently small. Hence $X$ is smooth for $t>t_0$
and since by the first claim $t_0$ can be chosen arbitrarily close to $0$,
the theorem is proved.

We prove the first claim. By the energy inequality,
\[
  \sum_{n=1}^\infty\int_0^t \bigl(\lambda_n X_n(s)\bigr)^2\,ds<\infty,
\]
hence $\sup_{n\geq1}(\lambda_n X_n(t))<\infty$ for a.~e.~$t>0$. Let $t_0>0$
be one of these times and set $K_0 = \sup_{n\geq1}\lambda_n^{\beta-2+\epsilon} X_n(t_0)$,
where $\epsilon$ is the parameter which has been set in the proof of Lemma
\ref{l:invariant}. Let
\[
  \bar Y_n(t)
    = \tfrac\delta{K_0} \lambda_n^{\beta-2+\epsilon} X_n(\tfrac\delta{K_0} t),
    \qquad n\geq1,\ t\geq t_0,
\]
where $\delta$ is the constant from Lemma~\ref{l:invariant}. It turns out that
$(\bar Y_n)_{n\geq1}$ is solution to~\eqref{e:Yeq} but with viscosity
$\bar\nu=\tfrac\delta{K_0}\nu$. Uniqueness of $(X_n)_{n\geq1}$ clearly
ensures uniqueness of $(\bar Y_n)_{n\geq1}$ for equation~\eqref{e:Yeq}
and so it is standard to show that the solutions $(\bar Y_n^{(N)})_{n\geq1}$
of~\eqref{e:YNeq} (with viscosity $\bar\nu$) converge to $(\bar Y_n)_{n\geq1}$.
Clearly $\sup_{n\leq N}\bar Y_n^{(N)}(t_0)\leq\delta$ for all $N\geq1$, therefore
Lemma~\ref{l:invariant} ensures that $Y_n^{(N)}(t)\leq1$ and in turns
$\lambda^{\beta-2+\epsilon} X_n(t)\leq\tfrac{K_0}{\delta}$ for
all $n\geq1$ and $t\geq t_0$. The proof of the first claim is complete.

We finally prove the second claim. Let $V_n = X_n\e^{\nu\lambda_n(t-t_0)}$,
then to prove smoothness of $X$ in a small interval, it is sufficient to show
that $V$ is bounded (uniformly in $n$) in the same interval. A direct
computation shows that
\[
\begin{aligned}
  \dot V_n
    &= - \nu(\lambda_n^2-\lambda_n)V_n
      + \lambda_{n-1}^\beta V_{n-1}^2
      - \lambda_n^\beta \e^{-\nu\lambda_{n+1}(t-t_0)} V_n V_{n+1}\\
    &\leq - \tfrac\nu2\lambda_n^2 V_n
         + \lambda_{n-1}^\beta V_{n-1}^2,
\end{aligned}
\]
so by comparison for ordinary differential equations we have that
$V_n(t)\leq \wV_n(t)$ for all $t\geq t_0$ for which $\wV$ is finite, where
$\wV$ is the solution to
\[
  \dot \wV_n
    =  -\tfrac\nu2\lambda_n^2 \wV_n + \lambda_{n-1}^\beta \wV_{n-1}^2,
\]
with initial condition $\wV_n(t_0) = V_n(t_0)$. Since by assumption the quantity
\[
  \sup_n(\lambda_n^{\beta-2+\epsilon}\wV_n(t_0))
    = \sup_n(\lambda_n^{\beta-2+\epsilon}V_n(t_0))
\]
is bounded, it follows that $\wV(t_0)\in\mathcal{W}_\epsilon$, where
$\mathcal{W}_\epsilon$ has been defined in \eqref{e:Wspace}. Following the same
lines of Remark~\ref{r:localsmooth}, one can apply Banach's fixed point
theorem to $\wV$ in the space $\mathcal{W}_\epsilon$ to show existence of a
solution in a small time interval.
\end{proof}
%%
%%
%%
%%%%%%%%%%%%%%%%%%%%%%%%%%%%%%%%%%%%%%%%%%%%%%
\section{The inviscid limit}\label{s:inviscid}

Following~\cite{BarFlaMor09b}, we give the following definitions of
solution.
\begin{definition}
A solution on $[0,T)$ (global if $T = \infty$) of~\eqref{e:dyadic_inviscid}
is a sequence $X = (X_n)_{n\geq1}$ of functions such that $X_n\in C^1([0,T);\R)$
for all $n\geq1$ and \eqref{e:dyadic_inviscid} is satisfied.

A \emph{Leray--Hopf} solution is a weak solution such that $X(t)\in H$ (where
$H$ is defined in~\eqref{e:Hspace}) and the energy inequality
\[
  \|X(t)\|_H \leq \|X(s)\|_H
\]
holds for all $s\geq0$ and $t\geq s$.
\end{definition}
We give a short summary of known facts on solutions to \eqref{e:dyadic_inviscid}.
\begin{itemize}
  \item There is at least one global in time Leray--Hopf solutions for all
    initial conditions in $H$ (see \cite{CheFriPav07}, the proof is given for
    $\beta=\tfrac52$ but the extension to all $\beta$ is straightforward).
  \item There is a unique local in time solution for ``regular'' enough initial
    conditions \cite{FriPav04a}.
  \item If the initial condition $(x_n)_{n\geq1}$ is \emph{positive}, then
    every weak solution is a Leray--Hopf solution and stays positive for
    all times \cite{BarFlaMor09b}.
  \item If $\beta\leq1$, there is a unique Leray--Hopf solution for every positive
    initial condition~\cite{BarFlaMor09b}.
  \item No positive solution can be smooth for all times. In \cite{CheFriPav07}
    they prove that, if $\beta=\tfrac52$, then the quantity $\lambda_n^{5/6}X_n(t)$
    cannot be bounded for all times.
\end{itemize}
\bigskip

We first start by giving a uniqueness criterion, based again on the idea
in \cite{BarFlaMor09b}.
\begin{lemma}[Uniqueness]
Given $T>0$, let $X = (X_n)_{n\geq1}$ be a \emph{positive} solution
to \eqref{e:dyadic_inviscid} on $[0,T]$.
\begin{itemize}
  \item If the quantity
    \begin{equation}\label{e:quniq1}
      \sup_{t\in [0,T],n\geq1}\bigl(\lambda_n^{\beta-1}X_n(t)\bigr)
    \end{equation}
    is finite, then $X$ is the unique solution with initial condition $(X_n(0))_{n\geq1}$
    in the class of Leray--Hopf solutions.
  \item If for some $\epsilon>0$ the quantity
    \begin{equation}\label{e:quniq2}
      \sup_{t\in[0,T]}\sup_{n\geq1}\Bigl(\lambda_n^{\frac13(\beta-1)+\epsilon}X_n(t)\Bigr)
    \end{equation}
    is finite, then $X$ is the unique solution with initial condition
    $(X_n(0))_{n\geq1}$ in the class of solutions satisfying \eqref{e:quniq2}.
\end{itemize}
\end{lemma}
\begin{proof}
We follow the same lines (with the same notation) of the proof of
Proposition~\ref{p:viscous_uniq}. Denote by $c_0$ the quantity~\eqref{e:quniq1}.
Let $Y=(Y_n)_{n\geq1}$ be another solution with the same initial condition
of $X$. Then for $N\geq1$,
\[
  \begin{aligned}
    \frac{d}{dt}\psi_N(t)
      &\leq -\frac12 \lambda_N^{\beta-1} Z_N Z_{N+1} W_N\\
      &\leq \frac12 \lambda_N^{\beta-1} (X_{N+1}Y_N^2 + X_N^2 Y_{N+1})\\
      &\leq c_0(X_N^2 + Y_N^2 + Y_{N+1}^2),
  \end{aligned}
\]
and so by integrating in time,
\[
  \psi_N(t)
    \leq c_0\int_0^t (X_N^2 + Y_N^2 + Y_{N+1}^2)\,ds.
\]
Since $X$ and $Y$ are both Leray--Hopf solutions, the right hand side in the
above inequality converges to $0$ as $N\to\infty$ and in conclusion
$\psi_n(t)=0$ for all $t\geq0$ and all $n\geq1$.

For the second statement, let $X$, $Y$ two solutions in the class, that is
with \eqref{e:quniq2} finite for both $X$ and $Y$. As in the proof of the
previous claim,
\[
  \frac{d}{dt}\psi_N(t)
    \leq \frac12 \lambda_N^{\beta-1} (X_{N+1}Y_N^2 + X_N^2 Y_{N+1})\\
    \leq c\lambda_N^{-3\epsilon},
\]
and so $\psi_N(t)\leq\lambda_N^{-3\epsilon}t$, which implies that $X=Y$.
\end{proof}
\begin{proof}[Proof of Theorem~\ref{t:main_inviscid}]
Assume that $\beta=\tfrac52$ and that $\sup_n \lambda_n^\gamma x_n<\infty$
for some $\gamma>\beta-2=\tfrac12$. First, notice that the second statement
of the previous lemma ensures that there is at most one solution satisfying
\eqref{e:inviscid_bound}. So to show that the inviscid dynamics is
bounded in the scaling $\lambda_n^\gamma$, we proceed by showing that the
viscous dynamics is convergent as $\nu\to0$. This shows both statements of
the theorem at once.

Given $\nu>0$, let $(X_n^{[\nu]})_{n\geq1}$ be the solution to the viscous
problem~\eqref{e:dyadic_viscous}. Again by uniqueness, it is sufficient to
work on a finite interval of time $[0,T]$, with $T>0$. So we fix $T>0$.
We know by Theorem~\ref{t:main_viscous} (possibly taking, without loss
of generality, a smaller value of $\gamma$) that
\[
  C_0
    := \sup_{\nu>0}\sup_{t\in[0,T]}\sup_{n\geq1}\bigl(\lambda_n^\gamma X_n^{[\nu_k]}(t)\bigr)
    <\infty,
\]
where $C_0$ depends only on the initial condition. Moreover for every
$n\geq1$ and $\nu\leq 1$,
\[
  |\dot X_n^{[\nu]}|
    \leq   \nu\lambda_n^2 X_n^{[\nu]}
         + \lambda_{n-1}^\beta(X_{n-1}^{[\nu]})^2
         + \lambda_n^\beta X_n^{[\nu]} X_{n+1}^{[\nu]}
    \leq c_n(C_0)
\]
where $c_n$ is a number independent of $\nu\leq1$ (although it does depend on
$n$). Hence by the Ascoli--Arzel\`a theorem for each $n$ the family
$\{X_n^{[\nu]} : \nu\in(0,1]\}$ is compact in  $C([0,T];\R)$. By a diagonal
procedure, we can find a common sequence $(\nu_k)_{k\in\N}$ and a limit
point $(X_n^{[0]})_{n\geq1}$ such that $X_n^{[\nu_k]}\to X_n^{[0]}$ uniformly
on $[0,T]$ for every $n\geq1$. Clearly any limit point is positive, satisfies
the equations~\eqref{e:dyadic_inviscid} and the bound \eqref{e:inviscid_bound},
hence by the previous lemma there is only one limit point and
$X_n^{[\nu]}\to X_n^{[0]}$ uniformly as $\nu\downarrow0$.
\end{proof}
\begin{remark}
Clearly the family $(X^{[\nu]})_{\nu\leq1}$ has limit points also when
$\beta\neq\tfrac52$. Moreover all limit points are bounded in the scaling
$\lambda_n^{\beta-2}$ if $\beta\in(2,\tfrac52]$ by virtue of Lemma~\ref{l:invariant}.
The main limitation is that the uniqueness lemma does not apply.
\end{remark}
%%
%%
%%
%%%%%%%%%%%%%%%%%%%%%%%%%%%%
\bibliographystyle{amsplain}
\bibliography{ShellModels}

\providecommand{\bysame}{\leavevmode\hbox to3em{\hrulefill}\thinspace}
\providecommand{\MR}{\relax\ifhmode\unskip\space\fi MR }
\providecommand{\MRhref}[2]{%
  \href{http://www.ams.org/mathscinet-getitem?mr=#1}{#2}
}
\providecommand{\href}[2]{#2}
\begin{thebibliography}{10}

\bibitem{BarFlaMor08}
David Barbato, Franco Flandoli, and Francesco Morandin, \emph{Energy
  dissipation and self-similar solutions for an unforced inviscid dyadic
  model}, 2008, to appear on Trans. Amer. Math. Soc. (arXiv: 0811.1689).

\bibitem{BarFlaMor09b}
\bysame, \emph{Well posedness for positive dyadic model}, 2009, to appear on
  Comptes Rendus Acad. Sci. Math. (arXiv: 0910.4995).

\bibitem{BarFlaMor09a}
\bysame, \emph{Uniqueness for a stochastic inviscid dyadic model}, Proc. Amer.
  Math. Soc. \textbf{138} (2010), no.~7, 2607--2617.

\bibitem{Che08}
Alexey Cheskidov, \emph{Blow-up in finite time for the dyadic model of the
  {N}avier-{S}tokes equations}, Trans. Amer. Math. Soc. \textbf{360} (2008),
  no.~10, 5101--5120.

\bibitem{CheFri09}
Alexey Cheskidov and Susan Friedlander, \emph{The vanishing viscosity limit for
  a dyadic model}, Phys. D \textbf{238} (2009), no.~8, 783--787.

\bibitem{CheFriPav07}
Alexey Cheskidov, Susan Friedlander, and Nata{\v{s}}a Pavlovi{\'c},
  \emph{Inviscid dyadic model of turbulence: the fixed point and {O}nsager's
  conjecture}, J. Math. Phys. \textbf{48} (2007), no.~6, 065503, 16.

\bibitem{FriPav04a}
Susan Friedlander and Nata{\v{s}}a Pavlovi{\'c}, \emph{Blowup in a
  three-dimensional vector model for the {E}uler equations}, Comm. Pure Appl.
  Math. \textbf{57} (2004), no.~6, 705--725.

\bibitem{FriPav04b}
\bysame, \emph{Remarks concerning modified {N}avier-{S}tokes equations},
  Discrete Contin. Dyn. Syst. \textbf{10} (2004), no.~1-2, 269--288, Partial
  differential equations and applications.

\bibitem{FujKat64}
Hiroshi Fujita and Tosio Kato, \emph{On the {N}avier-{S}tokes initial value
  problem. {I}}, Arch. Rational Mech. Anal. \textbf{16} (1964), 269--315.

\bibitem{KatPav05}
Nets~Hawk Katz and Nata{\v{s}}a Pavlovi{\'c}, \emph{Finite time blow-up for a
  dyadic model of the {E}uler equations}, Trans. Amer. Math. Soc. \textbf{357}
  (2005), no.~2, 695--708 (electronic).

\bibitem{KisZla05}
Alexander Kiselev and Andrej Zlato{\v{s}}, \emph{On discrete models of the
  {E}uler equation}, Int. Math. Res. Not. (2005), no.~38, 2315--2339.

\bibitem{MatSin99}
J.~C. Mattingly and Ya.~G. Sinai, \emph{An elementary proof of the existence
  and uniqueness theorem for the {N}avier-{S}tokes equations}, Commun. Contemp.
  Math. \textbf{1} (1999), no.~4, 497--516.

\bibitem{Wal06}
Fabian Waleffe, \emph{On some dyadic models of the {E}uler equations}, Proc.
  Amer. Math. Soc. \textbf{134} (2006), no.~10, 2913--2922 (electronic).

\end{thebibliography}
\end{document}